\documentclass[a4paper,11pt]{article}

\usepackage[top=3.2cm, bottom=3cm, left=3.45cm, right=3.45cm]{geometry}
\usepackage{scalerel}
\usepackage[T1]{fontenc}
\usepackage[utf8]{inputenc}
\usepackage{latexsym}
\usepackage{amsmath,amsthm,amssymb,amsfonts,hyperref}
\usepackage{mathtools}
\usepackage[english]{babel}
\usepackage{tikz}
\usetikzlibrary{patterns, matrix, shapes}
\usepackage{tikz-cd}
\usepackage{booktabs}
\usetikzlibrary{patterns, matrix}
\tikzstyle{NE-lines}=[pattern=north east lines, pattern color=black!45]
\usepackage{soul} 
\usepackage{url}

\usepackage[noblocks]{authblk}

\DeclareMathOperator{\img}{Im}
\newcommand{\etal}{et~al.}

\mathchardef\mhyphen="2D

\DeclareMathOperator{\Cay}{Cay}                
\DeclareMathOperator{\Sym}{\mathfrak{S}}       
\DeclareMathOperator{\I}{I}                    
\DeclareMathOperator{\Modinv}{\hat{I}}         
\DeclareMathOperator{\Ascseq}{A}               
\newcommand{\dA}{\Ascseq_d}                    
\DeclareMathOperator{\Modasc}{\hat{\Ascseq}}   
\DeclareMathOperator{\Self}{\widetilde{\Ascseq}}   
\DeclareMathOperator{\Selfinv}{\tilde{\I}}     
\newcommand{\dModasc}{\Modasc_d}               
\newcommand{\kAscseq}[1]{\Ascseq_{#1}}         
\newcommand{\kModasc}[1]{\hat{\Ascseq}_{#1}}   
\DeclareMathOperator{\F}{\mathfrak{F}}         
\DeclareMathOperator{\SMF}{\widetilde{\F}}         
\DeclareMathOperator{\RGF}{RGF}                

\DeclareMathOperator{\Bur}{Bur}                
\DeclareMathOperator{\WI}{WI}                  
\DeclareMathOperator{\fishpattern}{\mathfrak{f}} 
\DeclareMathOperator{\burget}{\mathtt{t}}        
\DeclareMathOperator{\sort}{sort}                
\newcommand{\beh}{\hat{\be}}
\newcommand{\R}{\mathcal{R}}
\newcommand{\selfpat}{\mathfrak{s}}

\DeclareMathOperator{\Asc}{Asc}          
\newcommand{\dAsc}{\Asc_d}               
\DeclareMathOperator{\asc}{asc}          
\newcommand{\dasc}{\asc_d}               
\DeclareMathOperator{\wDes}{wDes}        
\DeclareMathOperator{\wdes}{wdes}        
\DeclareMathOperator{\ides}{ides}        
\DeclareMathOperator{\nub}{nub}          
\DeclareMathOperator{\identity}{id}      
\DeclareMathOperator{\ltrmax}{lrMax}

\newcommand{\dhat}{\mathrm{hat}_d}
\newcommand{\khat}[1]{\mathrm{hat}_{#1}}

\newcommand{\al}{\alpha}
\newcommand{\be}{\beta}
\newcommand{\ga}{\gamma}
\newcommand{\alh}{\hat{\al}}
\DeclareMathOperator{\id}{id}
\DeclareMathOperator{\des}{des}

\newtheorem{theorem}{Theorem}[section]
\newtheorem{theorem*}{Theorem}[section]
\newtheorem{proposition}[theorem]{Proposition}
\newtheorem{lemma}[theorem]{Lemma}
\newtheorem{corollary}[theorem]{Corollary}

\newtheorem*{openproblem*}{Open Problem}

\theoremstyle{definition}

\newtheorem*{remark*}{Remark}

\newtheorem*{example*}{Example}

\begin{document}
\title{Self-modified difference ascent sequences}
\author[1]{Giulio Cerbai\thanks{G.C. is member of the Gruppo Nazionale Calcolo Scientifico--Istituto Nazionale di Alta Matematica (GNCS-INdAM).}}
\author[1]{Anders Claesson}
\author[2]{Bruce E. Sagan}
\affil[1]{Department of Mathematics, University of Iceland,
Reykjavik, Iceland, \texttt{akc@hi.is}, \texttt{giulio@hi.is}.}
\affil[2]{Department of Mathematics, Michigan State University,
East Lansing, MI 48824-1027, USA, \texttt{sagan@math.msu.edu}}
\date{}
\maketitle

\begin{abstract}
Ascent sequences play a key role in the combinatorics of Fishburn
structures. Difference ascent sequences are a natural generalization
obtained by replacing ascents with $d$-ascents.
We have recently extended the so-called hat map to difference ascent
sequences, and self-modified difference ascent sequences are the
fixed points under this map.
We characterize self-modified difference ascent sequences and
enumerate them in terms of certain generalized Fibonacci polynomials.
Furthermore, we describe the corresponding subset of $d$-Fishburn
permutations.
\end{abstract}

\thispagestyle{empty} 

\section{Introduction}

Ascent sequences were introduced by Bousquet-M\'{e}lou, Claesson, Dukes,
and Kitaev in 2010~\cite{BMCDK:2fp}. 
They are nonnegative integer sequences whose growth of their
entries is bounded by the number of ascents contained in the corresponding
prefix. In the same paper, modified ascent sequences were defined as the
bijective image of ascent sequences under a certain hat map.
Since then, both structures have been studied in their own
right~\cite{BP:asc,CMS:res,Cer24,Cer24b,CDDDS:asc,CCEPG:pat,DuS:pat,FJLY:new,
GK:num,JS:sep,KR:pasc,MS:enum,Yan:asc}.
Ascent sequences have an arguably more natural combinatorial definition,
which Bousquet-M\'{e}lou {\etal}~\cite{BMCDK:2fp} leveraged to show that
these objects are enumerated by the elegant power series
\begin{equation*}
  \sum_{n\geq 0}\prod_{k=1}^n\bigl(1-(1-x)^k\bigr)
\end{equation*}
whose coefficients are the Fishburn numbers (A022493 in the OEIS~\cite{oeis}).
Additionally, ascent sequences bijectively encode a recursive
construction of so-called Fishburn permutations, the latter defined
by avoidance of a certain bivincular pattern of length three.
On the other hand, modified ascent sequences interact more
transparently with interval orders~\cite{BMCDK:2fp} and
Fishburn trees~\cite{CC:trees}. They also have the advantage of being
Cayley permutations, i.e.\ sequences where all the integers between one
and the maximum value appear at least once.
Modified ascent sequences and Fishburn permutations are related by the
Burge transpose~\cite{CC:tpb}.

Different variants and generalizations of ascent sequences have been
proposed. Of particular interest to us are the weak ascent
sequences~\cite{BCD:was} and, more generally, the difference $d$-ascent
sequences recently introduced by Dukes and Sagan~\cite{DS:das}.
These are obtained by relaxing the bound on the growth of the
rightmost entry, which is achieved by replacing ascents with difference $d$-ascents.
For ease in reading, we will usually suppress the term ``difference" and just write $d$-ascent sequences, $d$-ascents, and so forth.
In our companion paper~\cite{CCS:mod}, we have generalized the hat map
to an injection that acts on $d$-ascent sequences, the $d$-hat,
obtaining modified $d$-ascent sequences in the process.
We have also studied some combinatorial properties of plain and modified
$d$-ascent sequences, as well as of the $d$-hat, exploring their
interplay with inversion sequences and the so-called difference $d$-Fishburn
permutations~\cite{ZZ:dperm}. Namely, we proved that the Burge transpose
maps bijectively modified $d$-ascent sequences to $d$-Fishburn
permutations.

The main objective of this paper is twofold, generalizing two theorems in Section 4.4 of~\cite{BMCDK:2fp}. First, we characterize and
enumerate the set $\Self_d$ of $d$-ascent sequences that are
invariant under the $d$-hat map, the self-modified $d$-ascent sequences.
As far as the characterization, we prove in Theorem~\ref{dself_char} 
that $\Self_d$ is the intersection
of $d$-ascent sequences and modified $d$-ascent sequences,
and thus $\Self_d$ is the largest set it could possibly be.
For the enumerative aspect, in Theorem~\ref{self-gf} we provide a two-variable generating
function for $\Self_d$:
$$
\Self_d(q,x) = 1 + qxF^{!}_d(x,qx)
$$
where $q$ tracks the length, $x$ tracks the maximum value, and
$F^{!}_d(x,y)$ are polynomials whose recursive definition is inspired
by the Fibonacci numbers and the Fibotorial (Fibonacci factorial).
Our second goal is accomplished in Theorem~\ref{thm_dselfperm}: we characterize by pattern
avoidance the set of $d$-Fishburn permutations corresponding
to self-modified $d$-ascent sequences, extending a characterization
by barred patterns given by Bousquet-M\'{e}lou {\etal}~\cite{BMCDK:2fp}
for the case $d=0$.

\section{Preliminaries}\label{section_intro}

Let $n$ be a nonnegative integer and let $\al:[n]\to [n]$ be an
endofunction, where $[n]=\{1,2,\ldots,n\}$.
For succinctness, we may identify
$\al$ with the word $\al=a_1\ldots a_n$, where $a_i=\al(i)$
for each $i\in[n]$. An index $i\in [n]$ is an \emph{ascent} of~$\al$ if
$i=1$ or $i\ge2$ and $a_i>a_{i-1}$. We define the {\em ascent set} of
$\al$ by
\[
  \Asc\al=\{i\in[n] \mid \text{$i$ is an ascent of $\al$}\}
\]
and we let $\asc\al=\#\Asc\al$ denote the number of ascents in $\alpha$.
Our conventions differ from some others in the literature in that we are
taking the indices of ascent tops, rather than bottoms, and that the
first position is always an ascent.

We call $\al$ an \emph{ascent sequence of length $n$} if for all $i\in[n]$ we have
\[
  a_i \le 1 + \asc(a_1\ldots a_{i-1}).
\]
In particular, $a_1\le 1+\asc\epsilon=1$, where~$\epsilon$
denotes the empty sequence. Since the entries of $\al$ are positive
integers, this forces $a_1=1$. 
We let
$$
\Ascseq=\{\al \mid \text{$\al$ is an ascent sequence}\}.
$$
As an example, the ascent sequences of length $3$ are
$$
111, 112, 121, 122, 123.
$$
%

Let $d$ be a nonnegative integer. An index $i\in[n]$ is a
\emph{$d$-ascent} if $i=1$ or $i\ge2$ and
$$
a_i>a_{i-1} - d.
$$
As with ordinary ascents, we define the {\em $d$-ascent set} of $\al$ by
$$
\dAsc\al=\{ i\in [n] \mid \text{$i$ is a $d$-ascent of $\al$}\}.
$$
and we let $\dasc\al=\#\dAsc\al$ denote the number of $d$-ascents in $\al$.
Note that a $0$-ascent is simply an ascent, while a $1$-ascent
($a_i>a_{i-1}-1$ or equivalently $a_i\ge a_{i-1}$) is what
is called a {\em weak ascent}.

We call $\al$ a \emph{$d$-ascent sequence} if for all $i\in[n]$ we have
$$
a_i \le 1 + \asc_d(a_1\ldots a_{i-1}).
$$
Once again, the above restriction forces $a_1=1$. 
Let
$$
\kAscseq{d,n} =\{\al \mid \text{$\al$ is a $d$-ascent sequence of length $n$}\}
$$
and
$$
\dA =\biguplus_{n\ge0} \Ascseq_{d,n}.
$$

Clearly, for $d=0$ we recover the set of ascent sequences,
while for $d=1$
we obtain the set of \emph{weak ascent sequences} of B\'enyi,
Claesson, and Dukes~\cite{BCD:was}.

Given an ascent sequence $\al$ we can form the corresponding
\emph{modified ascent sequence}~\cite{BMCDK:2fp} as follows: Scan the
ascents of $\al$ from left to right; at each step, every element
strictly to the left of and weakly larger than the current ascent is
incremented by one. More formally, let
$$
M(\al,j)=\al^+,\text{ where }\al^+(i)=a_i+
\begin{cases}
1 & \text{if $i<j$ and $a_i\geq a_{j}$,}\\
0 & \text{otherwise,}
\end{cases}
$$
and extend the definition of $M$ to multiple indices $j_1$,$j_2$,\dots,$j_k$
by
$$
M(\al,j_1,j_2,\dots,j_k)= M\bigl(M(\al,j_1,\dots,j_{k-1}),j_k\bigr).
$$
Then
$$
\alh=M(\al,\Asc\al),
$$
where in this context $\Asc\al$ is the ascent list of $\al$, that is, the ascent set of $\al$
ordered in increasing order.
For example, consider $\al=121242232$. The ascent set of
$\al$ is $\{1,2,4,5,8\}$; the ascent list is
$\Asc\al=(1,2,4,5,8)$, and we compute
$\alh$ as follows, where at each stage the entry governing the
modification is underlined while the entries which are modified are
bold:
\begin{align*}
  \al              &= 121242232;\\
  M(\al,1)         &= \underline{1}21242232;\\
  M(\al,1,2)       &= 1\underline{2}1242232;\\
  M(\al,1,2,4)     &= 1\mathbf{3}1\underline{2}42232;\\
  M(\al,1,2,4,5)   &= 1312\underline{4}2232;\\
  M(\al,1,2,4,5,8) &= 1\mathbf{4}12\mathbf{5}22\underline{3}2=\alh.
\end{align*}
Let
$$
\Modasc = \{\alh \mid \al\in\Ascseq\}.
$$
The construction described above can easily be inverted since
$\Asc\al=\Asc\alh$. In other words, the mapping $\Ascseq\to\Modasc$
by $\al\mapsto\alh$ is a bijection.

We~\cite{CCS:mod} have recently extended the ``hat map'',
$\al\mapsto\alh$, to $d$-ascent sequences: Let $d\ge 0$ and
$\al\in\kAscseq{d}$. Define \emph{$d$-hat} of $\al$ as
$$
\dhat(\al)=M\bigl(\al,\dAsc\al\bigr),
$$
where $\dAsc\al$ is the $d$-ascent list of~$\al$
obtained by putting the set in increasing order, and let
$$
\dModasc=\dhat(\dA)
$$
denote the set of \emph{modified $d$-ascent sequences}. As a special
case, $\khat{0}(\al)=\alh$, for each $\al\in\Ascseq_0$, and $\Modasc_0$
coincides with the set of modified ascent sequences defined by
Bousquet-M\'{e}lou~{\etal}~\cite{BMCDK:2fp}.
We will also use $\alh$ for $\dhat(\al)$ if $d$ is clear from context.

One may alternatively define the set
$\dModasc$ recursively~\cite{CCS:mod}:
Let $d\ge 0$ be a nonnegative integer. Let $\kModasc{d,0}=\{\epsilon\}$
and $\kModasc{d,1}=\{1\}$. Suppose $n\ge 2$ and write
$\al\in A_{d,n}$ as a concatenation
$$
\al=\be a
$$
where $a$ is the last letter of $\al$ and $\be$ is the corresponding prefix.  We also let $b$ denote the last letter of $\be$, which is also the last letter of $\beh$.
 Then
$\alh\in\kModasc{d,n}$ is of one of two forms depending on whether the
$a$ forms a $d$-ascent with the penultimate letter of $\al$:
\begin{itemize}
\item $\alh = \beh a$ \, if\, $1\leq a \leq b-d$, or
\item $\alh = \beh^+ a$ \, if\, $b-d< a \leq 1+\max\beh$,
\end{itemize}
where $\beh^+$ is defined by increasing every element of $\beh$ which is at least $a$ by one.
When there is no risk of ambiguity, we use ``+'' as a superscript that
denotes the operation of adding one to the entries greater than or equal
to~$a$ of a given sequence, where $a$ is a threshold determined by the
context. For instance, in $\beh^+a$ we let $\beh^+$ denote the sequence
obtained by adding one to each entry of $\beh$ that is greater than
or equal to $a$.


Let us recall a few more
definitions and results from~\cite{CCS:mod} that we shall need below.
An endofunction $\al=a_1\ldots a_n$ is an \emph{inversion sequence} if
$a_i\leq i$ for each $i\in [n]$. Let $\I_n$ denote the set of
inversion sequences of length $n$ and let $\I=\cup_{n\geq 0} \I_n$. We showed that
\begin{align}
  \I &= \bigcup_{d\ge 0}\dA, \nonumber
  \intertext{where the union is not disjoint,
  and we defined the set~$\Modinv$ of \emph{modified inversion sequences} by}
  \Modinv &= \bigcup_{d\ge 0}\dModasc. \label{eq_def_modinv}
\end{align}

A \emph{Cayley permutation} is an endofunction $\al$ where
$\img\al=[k]$ for some $k\le n$. Thus, a nonempty endofunction $\al$ is a Cayley
permutation if it contains at least one copy of each integer between~$1$
and its maximum element. The set of Cayley permutations of length $n$ is
denoted by $\Cay_n$.

\begin{proposition}[\cite{CCS:mod}]\label{misc}
  Given $d\ge 0$, let $\al\in\dA$ and let $\alh=\dhat(\al)$. Then
  \begin{itemize}
  \item[(a)] $\alh$ is a Cayley permutation,
  \item[(b)] $\dAsc\al=\nub\alh$,
  \item[(c)] $\dAsc\alh\subseteq\nub\alh$,
  \end{itemize}
  where 
  $$
  \nub\al = \{\min \al^{-1}(j) \mid 1\leq j\leq \max\al \}
  $$
  is the set of positions of leftmost occurrences.\footnote{
  The word ``nub'' comes from a Haskell function that removes duplicate
elements from a list, keeping only the first occurrence of each
element. It is also short for ``not used before.''}
\end{proposition}


\section{Self-modified \texorpdfstring{$d$}{d}-ascent sequences}\label{section_selfmod}

Recall~\cite{BMCDK:2fp} that an ascent sequence $\al$ is \emph{self-modified}
if $\al=\khat{0}(\al)$. We extend this notion to $d$-ascent sequences by
saying that $\al\in\dA$ is \emph{$d$-self-modified} if $\dhat(\al)=\al$.
Let
$$
\Self_d=\{\al\in\dA\mid\dhat(\al)=\al\}
$$
denote the set of $d$-self-modified $d$-ascent sequences; or, in short,
self-modified $d$-ascent sequences.
The main goal of this section is to obtain a better understanding of
the sets $\Self_d$. Namely, in Theorem~\ref{dself_char} we prove that
$$
\Self_d=\dA\cap\dModasc.
$$
Note that this is the largest set $\Self_d$ could possibly be.
Further, we prove in Theorem~\ref{dself_zero} that
$\Self_{d+1}\subseteq\Self_d$ for each $d\ge 0$, from which
the curious chain of inclusions
\begin{equation}\label{eq_chain}
\cdots\subseteq
\Self_2\subseteq\Self_1\subseteq\Self_0
\subseteq\kAscseq{0}\subseteq\kAscseq{1}\subseteq\kAscseq{2}
\subseteq\cdots
\end{equation}
follows. We start with a couple of simple lemmas.

Let $\al=a_1\ldots a_n$ be an endofunction. An index $i\in [n]$ is
a \emph{left-right maximum} of~$\al$ if $a_i>a_j$ for each $j\in[i-1]$.
We will use the notation
$$
\ltrmax\al = \{i\in [n] \mid \text{$i$ is a left-right maximum of $\al$}\}.
$$
It is easy to see that
\begin{equation}\label{eq_lrmax}
\ltrmax\al\subseteq\nub\al
\quad\text{and}\quad
\ltrmax\al\subseteq\Asc_0\al,
\end{equation}
two facts we will use often in this section.

\begin{lemma}\label{ltrmax_is_dasc}
For every $\al\in\I$ and $d\ge 0$, we have
$\ltrmax\al\subseteq\dAsc\al$.
\end{lemma}
\begin{proof}
By equation~\eqref{eq_lrmax}, we have
\begin{equation}
 \label{lrM:Ascd}
\ltrmax\al\subseteq\Asc_0\al\subseteq\dAsc\al,
\end{equation}
where the last set containment follows by definition of $d$-ascent.
\end{proof}

\begin{lemma}\label{lemma_be_behplus}
Let $\be\in\kAscseq{d,n}$ and let $\beh=\dhat(\be)$. Fix $a\ge1$. Let $\beh^+$ be
obtained by increasing by one each entry of~$\beh$ greater than or
equal to~$a$. If $\be=\beh^+$, then $\be=\beh$.
\end{lemma}
\begin{proof}
For $i\in[n]$, denote by~$b_i$, $b'_i$ and $b''_i$ the
$i$th entry of~$\be$, $\beh$ and~$\beh^+$, respectively.
By definition of~$\beh$ and~$\beh^+$, we have
$$
b_i \le b'_i \le b''_i.
$$
If $\be=\beh^+$, then $b_i=b'_i=b''_i$ for every~$i\in [n]$,
and our claim follows.
\end{proof}

\begin{proposition}\label{selfmod_iff}
Let $\al\in\dA$ and let $\alh=\dhat(\al)$. The following five statements are equivalent:
\begin{itemize}
\item[(a)] $\alh=\al$.
\item[(b)] $\dAsc\al\subseteq\ltrmax\al$.
\item[(c)] $\dAsc\al\subseteq\nub\al$.
\item[(d)] $\dAsc\al=\ltrmax\al$.
\item[(e)] $\dAsc\al=\nub\al$.
\end{itemize}
Furthermore, if $\al\in\Self_d$ then $\dAsc\al=\Asc_0\al$.
\end{proposition}
\begin{proof}
Let us start by proving (a)$\,\Leftrightarrow\,$(b) and (a)$\,\Leftrightarrow\,$(c).

We prove (a)$\,\Leftrightarrow\,$(b) by induction on the length of~$\al$.
Our claim holds if~$\al$ has length at most one. Let $\al\in\dA$ and
suppose that $\al=\be a$, for some $\be\in\kAscseq{d,n}$,
where $1\le a\le 1+\asc_d\be$ and $n\ge 1$. Let $\beh=\dhat(\be)$.
We shall consider two cases, according to whether
or not~$a$ forms a $d$-ascent with the last letter~$b$ of~$\be$:
$$
\alh=\begin{cases}
\beh a, & \text{if $1\le a\le b-d$};\\
\beh^+ a, & \text{if $b-d<a\le 1+\asc_d\be$}.
\end{cases}
$$
Initially, suppose that $\al=\alh$. We show that
$\dAsc\al\subseteq\ltrmax\al$. If $a\le b-d$, then $\be=\beh$
since $\al=\alh$. Moreover,
$$
\dAsc\al=\dAsc\be\subseteq\ltrmax\be=\ltrmax\al,
$$
where we used our induction hypothesis on $\be$.
Otherwise, suppose that $b-d<a$. Since we assumed $\al=\alh$, we
have $\be a=\beh^+a$, from which $\be=\beh^+$ follows.
By Lemma~\ref{lemma_be_behplus}, we have $\be=\beh$.
Furthermore, since $\beh=\beh^+$, no entry of~$\beh$ is increased by one
in going from~$\beh$ to~$\beh^+$. Thus, $a>c$ for each entry~$c$
of $\beh=\be$, i.e.\ $n+1\in\ltrmax\al$. Finally,
$$
\dAsc\al=\dAsc\be\uplus\{n+1\}\subseteq\ltrmax\be\uplus\{n+1\}=\ltrmax\al,
$$
where $\dAsc\be\subseteq\ltrmax\be$ by our induction hypothesis.

To prove the converse, suppose that
$\dAsc\al\subseteq\ltrmax\al$. Then
$$
\dAsc\be=\dAsc\al\cap[n]\subseteq\ltrmax\al\cap[n]=\ltrmax\be,
$$
and $\beh=\be$ follows by our induction hypothesis. Now, if $a\le b-d$, then
$$
\alh=\beh a = \be a=\al,
$$
as wanted. On the other hand, if $b-d<a$, then
$$
n+1\in\dAsc\al\subseteq\ltrmax\al.
$$
Since $n+1\in\ltrmax\al$ and $\beh=\be$, no entry of~$\beh$ is increased by one
in~$\beh^+$. 
Thus $\beh^+=\beh$, and $\alh=\al$ follows immediately.

Let us now prove (a)$\,\Leftrightarrow\,$(c). The proof is similar to the one
just given so we shall keep the same notation. If $\alh=\al$, then
\begin{equation}
\label{Asc:nub}
\dAsc\al=\nub\alh=\nub\al,
\end{equation}
where the first equality is item~(b) of Proposition~\ref{misc}.

On the other hand, suppose that $\dAsc\al\subseteq\nub\al$.
Then
$$
\dAsc\be=\dAsc\al\cap[n]\subseteq\nub\al\cap[n]=\nub\be.
$$
and $\beh=\be$ by our induction hypothesis. Now, if $1\le a\le b-d$, then
$\alh=\beh a=\be a=\al$ and we are done. Otherwise, $b-d<a\le 1+\asc_d\be$
and $\alh=\beh^+a=\be^+ a$. Note that $\be$ is a Cayley
permutation by item~(a) of Proposition~\ref{misc}. In particular,
we have $\img\be=[k]$, where $k=\max\be$.
Further, we have $n+1\in\dAsc\al\subseteq\nub\al$. That is, the last
entry~$a$ is a leftmost copy in $\al=\be a$. Since $\img\be=[k]$,
it must be $a\ge k+1$.
In particular, $\be^+=\be$ and thus $\alh=\be^+a=\be a=\al$, as wanted.

Next, we prove that we can we replace the set inclusions
with equalities in items (b) and (c).
Clearly, it is enough to show that $\alh=\al$ implies
$$
\dAsc\al=\ltrmax\al=\nub\al,
$$
where the inclusions $\dAsc\al\subseteq\ltrmax\al$ and
$\dAsc\al\subseteq\nub\al$ hold by what has been proved before.
The equality $\dAsc\al=\ltrmax\al$ now follows directly from
Lemma~\ref{ltrmax_is_dasc}. 
And $\dAsc\al=\nub\al$ was proved in~\eqref{Asc:nub}.

Finally, we prove that $\dAsc\al=\Asc_0\al$ if $\al\in\Self_d$.
Using equation~\eqref{lrM:Ascd} and item~(d), we obtain
$$
\ltrmax\al\subseteq\Asc_0\al\subseteq\dAsc\al=\ltrmax\al,
$$
from which the desired equality follows.
\end{proof}


%

We are now ready for the promised characterization of $\Self_d$.

\begin{theorem}\label{dself_char}
For each $d\ge 0$, we have $\Self_d=\dA\cap\dModasc$.
\end{theorem}
\begin{proof}
If $\al\in\Self_d$, then $\al\in\dA$ and $\al=\dhat(\al)\in\dModasc$
as well. On the other hand, suppose that $\al\in\dA\cap\dModasc$.
Since $\al\in\dModasc$, we have $\dAsc\al\subseteq\nub\al$ by
item~(c) of Proposition~\ref{misc}.
Then $\dhat(\al)=\al$ by Proposition~\ref{selfmod_iff}, i.e.\
we have $\al\in\Self_d$.
\end{proof}

We now have all the ingredients to prove equation~\eqref{eq_chain},
which is an immediate consequence of item~(a) of the following
theorem.

\begin{theorem}\label{dself_zero}
For every $d\ge 0$, we have:
\begin{itemize}
\item[(a)] $\Self_d\subseteq\Self_k$ for each $0\le k\le d$.
In particular, $\Self_d\subseteq\Self_0$ and
$\Self_{d+1}\subseteq\Self_d$.
\item[(b)] $\Self_d=\Self_0\cap\dModasc$.
\end{itemize}
\end{theorem}
\begin{proof}
%

(a) Let $\al\in\Self_d$ and let $0\le k\le d$.
We prove that $\al\in\Self_k$. By definition of $d$-ascent, we have
$$
\Asc_0\al\subseteq\Asc_k\al\subseteq\Asc_d\al.
$$
By Proposition~\ref{selfmod_iff}, we have $\Asc_0\al=\Asc_d\al$ and
$$
\Asc_0\al=\Asc_k\al=\Asc_d\al=\nub\al.
$$
In particular, $\al\in\kAscseq{k}$ since $\Asc_k\al=\Asc_d\al$
and $\al\in\kAscseq{d}$. Furthermore, using the same proposition once
more we have that $\al\in\Self_k$ since $\Asc_k\al=\nub\al$.

(b) We start with the inclusion
$\Self_d\subseteq\Self_0\cap\dModasc$. By
Theorem~\ref{dself_char},
$$
\Self_d=\dA\cap\dModasc\subseteq\dModasc.
$$
Note also that $\Self_d\subseteq\Self_0$ by item~(a) of
this theorem, from which the desired inclusion follows. To prove the
opposite inclusion, we can once again use Theorem~\ref{dself_char} to get
$$
\Self_0=\kAscseq{0}\cap\kModasc{0}
\subseteq\kAscseq{0}\subseteq\dA
$$
and
$$
\Self_0\cap\dModasc\subseteq\dA\cap\dModasc=\Self_d,
$$
which concludes the proof.
\end{proof}

The only inversion sequence that is $d$-self-modified for every $d\ge 0$
is the increasing permutation, as we will now show.

\begin{proposition}
We have
$$
\bigcap_{d\ge 0}\Self_d=\{ 12\ldots n\mid n\ge 0\}.
$$
\end{proposition}
\begin{proof}
It follows immediately from the inductive description of $\dhat$
that $12\ldots n\in\Self_d$ for each $n,d\ge 0$. That is,
$$
\{ 12\ldots n\mid n\ge 0\}\subseteq\bigcap_{d\ge 0}\Self_d.
$$
Conversely, suppose that $\al\in\Self_d$ for each $d\ge 0$.
Note first that necessarily $\al\in\I$.
For a contradiction, suppose that $\al\in\I$ is not an increasing
permutation; equivalently, let $\al=a_1\ldots a_n$ and suppose that there
is some $i\in[n-1]$ such that $a_i\ge a_{i+1}$. Note that
$i+1\in\Asc_n\al=[n]$ and $a_i\ge a_{i+1}$. Therefore, the entry $a_i$
is increased by one under the action of the map $\khat{n}$, which contradicts
$\al\in\Self_n$.
\end{proof}

Given $\al\in\I$, let us define 
the set
$$
H(\al)=\{\dhat(\al)\mid d\ge 0\text{ and }\al\in\dA\}
$$
of all the $d$-hats of~$\al$.
We~\cite[Corollary 4.5]{CCS:mod} proved that $H(\al)$ is finite by showing
that $\dhat(\al)=\khat{n}(\al)$ for each $\al\in\I$ and $d\ge n$,
where $n$ is the length of $\al$.
As an example, consider the inversion sequence $\al=11312$.
We see that $\al\in\kAscseq{d}$ if and only if $d\ge 1$. Further,
we have $\khat{1}(\al)=\khat{2}(\al)=31412$ and
$\khat{d}(\al)=43512$ for $d\ge 3$, hence $H(\al)=\{31412,43512\}$.

A seemingly more general notion of self-modified sequence arises
by defining an inversion sequence $\al$ to be
\emph{self-modified} if $\al\in H(\al)$. The set $\Selfinv$ of
\emph{self-modified inversion sequences} is defined accordingly as
$$
\Selfinv=\{\al\in\I\mid\al\in H(\al)\}.
$$
It turns out that self-modified inversion sequences coincide
with self-modified ascent sequences. Indeed, it is easy to see that
\begin{equation}\label{eq_invself}
\Selfinv=\bigcup_{d\ge 0}\Self_d=\Self_0,
\end{equation}
where the last equality follows by item~(a) of Theorem~\ref{dself_zero}.

To end this section, we wish to provide two alternative characterizations
of $\Self_0$ (see equations~\eqref{altchar1} and~\eqref{altchar2}).
Recall that an endofunction $\al=a_1\ldots a_n$
is a \emph{restricted growth function} if $a_1=1$ and
$$
a_{i+1}\le 1+\max(a_1\ldots a_i)
$$
for each $i\in[n-1]$.
We let
$$
\RGF =\{\al \mid \text{$\al$ is a restricted growth function}\}.
$$
Note that $\RGF\subseteq\kAscseq{0}$.
There is a standard bijection~\cite{Sagan:rgf} between RGFs and set partitions of $[n]$
where the leftmost copies correspond to the minima of the blocks (nonempty subsets of $[n]$).
Next, we show that restricted growth functions 
are Cayley permutations whose leftmost
copies appear in increasing order.

\begin{lemma}\label{RGF_char}
We have
$$
\RGF=\{\al\in\Cay\mid\nub\al=\ltrmax\al\}.
$$
\end{lemma}
\begin{proof}
We use induction on the length~$n$ of $\al$, where the cases $n=0$
and $n=1$ are trivial.
Let $n\ge 2$ and let $\al=a_1\ldots a_{n-1}a=\be a$, where
$\be=a_1\ldots a_{n-1}$ and $a\in [n]$. Initially, suppose
that $\al\in\RGF$. 

We first show that $\al\in\Cay$. By definition of RGF we have
$\be\in\RGF$. So by induction $\be\in\Cay$, say with image $[k]$. If
$a\le k$ then $\al$ has the same image. Otherwise the RGF condition
forces $a=k+1$ and $\al$ has image $[k+1]$. The proof that
$\nub\al=\ltrmax\al$ is similar: when $a\le k$ then both $\nub$ and
$\ltrmax$ do not change in passing from $\be$ to $\al$. And if $a=k+1$
then $n$ is added to both sets.

On the other hand, suppose that $\al\in\Cay$ and $\nub\al=\ltrmax\al$.
We show that $\al\in\RGF$. If $n\in\nub\al=\ltrmax\al$, then
\begin{align*}
[\max\al]&=\img\al
  && \text{(since $\al\in\Cay$)}\\
&=\img\be\uplus\{a\}
  && \text{(since $n\in\nub\al$)}
\end{align*}
and also $a=\max\al$ since $n\in\ltrmax\al$. Therefore, we have
$a=1+\max\be$ and $\img\be=[a-1]$. In particular, $\be\in\Cay$
and
$$
\nub\be=\nub\al\cap[n-1]=\ltrmax\al\cap[n-1]=\ltrmax\be.
$$
By induction, we have $\be\in\RGF$, and $\al\in\RGF$ follows as
well since $a=1+\max\be$. The case where
$n\notin\nub\al=\ltrmax\al$ can be proved in a similar fashion,
and we leave the details to the reader.
\end{proof}

Next we characterize self-modified $d$-ascent sequences as those
modified $d$-ascent sequences that are restricted growth functions.

\begin{proposition}\label{selfmod_rgf_char}
For each $d\ge 0$, we have
$$
\Self_d=\kModasc{d}\cap\RGF.
$$
\end{proposition}
\begin{proof}
Let us start with the inclusion
$\Self_d\subseteq\kModasc{d}\cap\RGF$. Let
$\al\in\Self_d$. Then $\al=\khat{d}(\al)\in\kModasc{d}$ and~$\al$ is a
Cayley permutation by item~(a) of Proposition~\ref{misc}. Furthermore,
by Proposition~\ref{selfmod_iff},
$$
\ltrmax\al=\dAsc\al=\nub\al,
$$
hence $\al\in\RGF$ follows by Lemma~\ref{RGF_char}.

To prove the remaining inclusion, recall that $\RGF\subseteq\kAscseq{0}$.
Thus
$$
\kModasc{d}\cap\RGF
\subseteq\kModasc{d}\cap\kAscseq{0}
\subseteq\kModasc{d}\cap\kAscseq{d}
=\Self_d,
$$
where the last equality is Theorem~\ref{dself_char}.
\end{proof}

Letting $d=0$ in Proposition~\ref{selfmod_rgf_char} yields
\begin{equation}\label{altchar1}
\Self_0=\kModasc{0}\cap\RGF.
\end{equation}

An alternative description of $\Self_0$ is
showed in the next result.

\begin{corollary}\label{invself_char}
We have
\begin{equation}\label{altchar2}
\Self_0=\Modinv\cap\RGF.
\end{equation}

\end{corollary}
\begin{proof}
We have
\begin{align*}
\Self_0
&=\bigcup_{d\ge 0}\Self_d
  && \text{(by equation~\eqref{eq_invself})}\\
&=\bigcup_{d\ge 0}\left(\kModasc{d}\cap\RGF\right)
  && \text{(by Proposition~\ref{selfmod_rgf_char})}\\
&=\bigl(\bigcup_{d\ge 0}\kModasc{d}\bigr)\cap\RGF\\
&=\Modinv\cap\RGF
  && \text{(by equation~\eqref{eq_def_modinv})},
\end{align*}
finishing the proof.
\end{proof}

Theorem~\ref{dself_char} characterizes self-modified $d$-ascent
sequences as $\Self_d=\dA\cap\dModasc$. It is easy to see that
the inclusion $\Selfinv\subseteq\I\cap\Modinv$ holds as well. Indeed,
using equation~\eqref{eq_invself},
$$
\Selfinv=\Self_0=\kAscseq{0}\cap\kModasc{0}
\subseteq\I\cap\Modinv.
$$
However, the opposite inclusion does not hold. For instance,
let us consider again the inversion sequence $\al=11312$.
Here, $\al=\khat{0}(11212)$ is a member of $\I\cap\Modinv$.
On the other hand, we have seen that $H(\al)=\{31412,43512\}$,
thus $11312\notin H(11312)$ and hence $11312\notin \Selfinv$.

We end this section with one more remark. Self-modified inversion
sequences are related to restricted growth functions through
$\Selfinv=\Modinv\cap\RGF$. But the inclusion $\RGF\subseteq\Modinv$ does
not hold: $1212\in \RGF$, but $1212\notin\Modinv$. One way to obtain an equality
would be to alter the recursive definition of $\dModasc$ given in
Section~\ref{section_intro} by allowing $d=-\infty$ and letting the last
letter~$a$ be chosen in the interval $[1+\max\be]$. This would give
$\kModasc{-\infty}=\RGF$.
We will keep the standard definition when $d$ is finite so as to correspond more closely to  the other literature on this topic and
 leave the direction  discussed above open for a future investigation.

\section{Enumeration of self-modified \texorpdfstring{$d$}{d}-ascent sequences}\label{section_enumdself}

We aim to determine the generating function
\[
  \Self_d(q,x) = \sum_{\alpha}q^{\max(\alpha)}x^{|\alpha|},
\]
where the sum ranges over all self-modified $d$-ascent sequences. In
particular, the coefficient of $x^n$ in $\Self_d(1,x)$ is
$|\Self_{d,n}|$, the number of self-modified $d$-ascent sequences of
length $n$. One reason it is natural to refine this count by the
distribution of $\max$ is that, by Proposition~\ref{selfmod_rgf_char},
any self-modified $d$-ascent sequence is a restricted growth function
and $\max$ gives the number of blocks in the corresponding set
partition. For instance, the first few terms for $d=2$ are
\[
  \Self_2(q,x) =
  \begin{aligned}[t]
    1 &\,+\, qx \,+\, q^{2}x^2 \,+\, q^{3}x^3 \,+\, (q^{3} +  q^{4})x^4 \\
      &\,+\, (3 q^{4} +  q^{5})x^5 \\
      &\,+\, (2 q^{4} + 6 q^{5} +  q^{6})x^6 \\
      &\,+\, (12 q^{5} + 10 q^{6} +  q^{7})x^7 \\
      &\,+\, (9 q^{5} + 39 q^{6} + 15 q^{7} +  q^{8})x^8 \\
      &\,+\, (2 q^{5} + 75 q^{6} + 95 q^{7} + 21 q^{8} +  q^{9})x^9 \,+\, \cdots
  \end{aligned}
\]
We shall derive an explicit expression for $\Self_d(q,x)$, for any
$d\geq 0$, in terms of certain Fibonacci polynomials which we introduce
below.

The \emph{Fibonacci numbers} $F_n$ are defined by the second order recurrence
relation $F_{n} = F_{n-1} + F_{n-2}$ with initial terms
$F_0=F_1=1$. Many generalizations of the Fibonacci numbers have been
proposed. One may for instance consider the $d$th order recurrence
relation $F_{n} = F_{n-1} + F_{n-d}$ with initial terms
$F_0=F_1=\cdots=F_{d-1}=1$. Or, generalizing in a different direction,
one may consider Fibonacci polynomials such as those given by
$F_0(x)=F_1(x)=1$ and $F_{n}(x) = F_{n-1}(x) + xF_{n-2}(x)$. Combining
these two ideas we define, for any fixed $d\geq 0$, and any integer $n$,
\begin{equation}
\label{Fdnx}
\left\{
\begin{aligned}
  F_{d,n}(x) &= 1 && \text{for $n < d$,} \\
  F_{d,n}(x) &= F_{d,n-1}(x) + xF_{d,n-d}(x) && \text{for $n\geq d$.}
\end{aligned}
\right.
\end{equation}
In generalizations of the Fibonacci sequence, like the one above, the
``smallest'' case typically corresponds to the (classical) Fibonacci
recurrence, which in our definition is $d=2$ (with $x=1$). Note that we,
however, also allow $d=0$ and $d=1$.

If $d=0$, then $F_{0,n}(x) = F_{0,n-1}(x) + xF_{0,n}(x)$, which together
with the initial condition $F_{0,-1}(x)=1$ gives
\[
  F_{0,n}(x) = 1/(1-x)^{n+1}=(1+x+x^2+\cdots)^{n+1}.
\]
In particular, $F_{0,n}(x)$ is a power series, while it is easy to see
that $F_{d,n}(x)$ is a polynomial for any $d\geq 1$. For example, when
$d=1$ we have $F_{1,0}(x)=1$ and $F_{1,n}(x) = (1+x)F_{1,n-1}(x)$ for
$n\geq 1$, and hence
\[
  F_{1,n}(x) = (1+x)^{n}.
\]
For any $d\geq 0$, define the generating function
\[
  F_{d}(x,y) = \sum_{n\geq 0} F_{d,n}(x)y^n.
\]
Using standard techniques for converting recursions into generating functions (see, for example~\cite[Section 3.6]{Sagan:aoc}) it follows from the recurrence relation for
$F_{d,n}(x)$ that
\begin{equation}\label{dfib-gf}
  F_{d}(x,y) = \frac{1}{1-y-xy^d}.
\end{equation}
Viewing this as a geometric series and applying the binomial theorem we
find that
\[
  F_{d}(x,y) = \sum_{m\geq 0}\sum_{k=0}^m\binom{m}{k}x^ky^{(d-1)k+m}
\]
and on extracting the coefficient of $y^n$ we get
\begin{equation*} 
  F_{d,n}(x) = \sum_{k\geq 0} \binom{n-(d-1)k}{k}x^k.
\end{equation*}
For reference, the first few polynomials for $d=2$ are
\begin{align*}
  F_{2,0}(x) &= F_{2,1}(x) = 1; \\
  F_{2,2}(x) &= 1+x; \\
  F_{2,3}(x) &= 1 + 2x; \\
  F_{2,4}(x) &= 1 + 3x + x^2; \\
  F_{2,5}(x) &= 1 + 4x + 3x^2; \\
  F_{2,6}(x) &= 1 + 5x + 6x^2 + x^3.
\end{align*}

Let us write $\mu\vDash n$ to indicate that $\mu$ is an integer
composition of $n$. A well-known interpretation of the $n$th Fibonacci
number, $F_n$, is the number of compositions $\mu\vDash n$ with parts in
$\{1,2\}$. Similarly, for $d\geq 2$, an interpretation of $F_{d,n}(1)$
is the number of compositions $\mu\vDash n$ with parts in $\{1,d\}$, and
the polynomial $F_{d,n}(x)$ records the distribution of $d$-parts in
such compositions. In symbols,
\[
  F_{d,n}(x) = \sum_{\substack{\mu=(m_1,\dots,m_k)\,\vDash\, n\\[0.5ex] m_i\in \{1,d\}}} x^{|\mu|_d},
\]
where $|\mu|_d = \#\{i: m_i = d\}$. This interpretation works for $d=0$
as well; as previously noted $F_{0,n}(x)$ is power series rather than a
polynomial in that case. For $d=1$ a little extra care is needed. For
our combinatorial interpretation to work we must regard the $1$-parts
as being distinct from the $d$-parts, even though $d=1$ in this case.
In other words, $F_{1,n}(x)$ records the
distribution of $1'$-parts in compositions $\mu\vDash n$ with
parts in $\{1,1'\}$, where $1$ and $1'$ denote
two different kinds of parts, both of size $1$.

Compositions with parts in $\{1,d\}$ are closely related to compositions
whose parts are of size at least $d$. To precisely state the relationship,
define
\begin{equation}
\label{KdnxDef}
  K_{d,n}(x) = \sum_{\mu}x^{\ell(\mu)}\quad\text{and}\quad K_d(x,y) = \sum_{n\geq 0} K_{d,n}(x)y^n,
\end{equation}
where the former sum ranges over all integer compositions $\mu$ of $n$
into parts of size $d$ or larger, and $\ell(\mu)$ denotes the number of
parts of $\mu$.
\begin{lemma}\label{lemma-K}
  We have
  \[
    K_d(x,y) = \frac{1-y}{1-y-xy^d}.
  \]
\end{lemma}
\begin{proof}
  The generating function for single parts of size at least $d$ is
  $y^d/(1-y)$. Thus
  \begin{align*}
    \sum_{n\geq 0} K_{d,n}(x) y^n
    &= \sum_{k\geq 0} \left(\frac{xy^d}{1-y}\right)^k \\
    &= \frac{1}{1-\cfrac{xy^d}{1-y}} = \frac{1-y}{1-y-xy^d}. \qedhere
  \end{align*}
\end{proof}

\begin{lemma}\label{K-F}
  For any $d\geq 0$, we have
  \begin{enumerate}
  \item [(a)] $F_{d,n}(x) = F_{d,n-1}(x) + K_{d,n}(x)$\, for $n\geq 1$, and
  \item [(b)] $F_{d,n}(x) = K_{d,0}(x)+K_{d,1}(x)+\cdots+K_{d,n}(x)$\, for $n\geq 0$.
  \end{enumerate}
\end{lemma}
\begin{proof}
  Identity (b) is obtained by repeated application of identity (a), so
  let us focus on (a). An immediate consequence of Lemma~\ref{lemma-K}
  and identity~\eqref{dfib-gf} is
  \begin{equation*}
    K_{d}(x,y) = (1-y)F_{d}(x,y),
  \end{equation*}
  from which (a) follows by identifying coefficients. 
\end{proof}

While the previous proof settles the claimed identity, let us also
provide a, perhaps more elucidating, combinatorial proof of~(a).
Assume $n\geq 1$ and let
$\mu=(m_1,m_2,\dots,m_k)\vDash n$ be such that $m_i\in \{1,d\}$ for
each $i\in[k]$. If $m_k=1$ then we map $\mu$ to the composition
$\mu'=(m_1,m_2,\dots,m_{k-1})\vDash n-1$ obtained from $\mu$ by
removing its last part. This procedure is trivially reversible. If
$m_k=d$, then we need to map $\mu$ (in a reversible way) to a
composition $\nu$ with parts of size at least $d$. Moreover, $\nu$
should have as many parts as $\mu$ has $d$-parts. Having spelled out
these criteria, the map now presents itself: Assume that $m_i$ is the
first $d$-part of $\mu$ (which exists, since $m_k=d$); that is,
$m_1=\cdots=m_{i-1}=1$ and $m_i=d$. We simply sum these up to get the
first part $n_1= i-1+d$ of $\nu$; the second part $n_2$ is obtained by
applying the same procedure to $(m_{i+1},m_{i+2},\dots,m_k)$; and so on.
For instance, $\mu=(1,3,3,1,1,3)$, where $d=3$ and $n=12$, gets mapped
to $\nu = (4,3,5)$.

In order to prove the expression for $\tilde{A}_d(q,x)$ in Theorem~\ref{self-gf} below, we will need the following definition
and characterization of the elements of $\dA$ which are self-modified.
Let us say that a sequence of numbers $c_1c_2\dots c_{\ell}$ is
\emph{decreasing with pace $d$} if the difference between consecutive
elements is at least $d$; that is, if
$c_j-c_{j+1} \geq d$ for $1\leq j<\ell$.

\begin{lemma}\label{self-mod-factoring}
  Any $\alpha\in\dA$ is self-modified if and only if it can be written
  \[
    \alpha = 1B_12B_2\ldots kB_k,
  \]
  where $k=\max\alpha$ and each factor $iB_i$ is decreasing with pace $d$.
\end{lemma}
\begin{proof}
Let $\alpha\in\Self_d$ with $\max \alpha = k$ be given. 
By Proposition~\ref{selfmod_rgf_char}, $\al$ is an RGF and so can be written
$$
\alpha = 1B_12B_2\ldots kB_k
$$
where $\nub\al$ is the set of positions of the elements $1,2,\ldots,k$ which are not in the $B_i$.
But from Proposition~\ref{selfmod_iff} we have $\dAsc\al=\nub\al$. Thus each factor $iB_i$ must be
void of $d$-ascents; that is, each $iB_i$ is decreasing with pace $d$.

  Conversely, if $\alpha = 1B_12B_2\ldots kB_k\in\kAscseq{d}$, where $k=\max\alpha$
  and each factor $iB_i$ is decreasing with pace $d$, then
  $\nub\alpha =\dAsc\alpha$ and hence $\alpha$ is
  self-modified by Proposition~\ref{selfmod_iff}.
\end{proof}

The last ingredient needed to state Theorem~\ref{self-gf} is an analogue of the factorial function for Fibonacci numbers and polynomials.
The $n$th \emph{Fibonacci factorial}, also called \emph{Fibonorial} or
\emph{Fibotorial}, is defined by
\[
  F^{!}_n = F_1F_2\cdots F_n.
\]
In this manner we also define
\[
  F^{!}_{d,n}(x) = \prod_{i=0}^n F_{d,i}(x).
\]
Note that the index $i$ ranges from 0 to $n$ rather than from 1 to
$n$. This only makes a difference when $d=0$ since $F_{0,0}(x)=1/(1-x)$
while $F_{d,0}(x)=1$ for $d\geq 1$. In particular,
\[
  F^{!}_{0,n}(x)
  = \left(\frac{1}{1-x}\right)^{\binom{n+2}{2}}
\quad\text{and}\quad
  F^{!}_{1,n}(x) = (1+x)^{\binom{n+1}{2}}.
\]
For $d\geq 2$ we do not have such simple formulas. As an illustration,
the first few polynomials for $d=2$ are
\begin{align*}
  F^{!}_{2,0}(x) &= F^{!}_{2,1}(x) = 1;\\
  F^{!}_{2,2}(x) &= 1 + x;\\
  F^{!}_{2,3}(x) &= 1 + 3x + 2x^2;\\
  F^{!}_{2,4}(x) &= 1 + 6x + 12x^2 + 9x^3 + 2x^4;\\
  F^{!}_{2,5}(x) &= 1+ 10x + 39x^2 + 75x^3 + 74x^4 + 35x^5 + 6x^6.
\end{align*}
For any $d\geq 0$, define the generating function
\[
  F^{!}_{d}(x,y) = \sum_{n\geq 0} F^{!}_{d,n}(x)y^n.
\]
The following theorem reveals a striking relationship between
Fibonacci factorials and the number of self-modified $d$-ascent
sequences. Recall the generating function
$\Self_d(q,x) = \sum_{\alpha}q^{\max(\alpha)}x^{|\alpha|}$, where the
sum ranges over all $\al\in\Self_d$.

\begin{theorem}\label{self-gf}
  For any $d\geq 0$,
  \[
    \Self_d(q,x) = 1 + qxF^{!}_d(x,qx).
  \]
\end{theorem}
\begin{proof}
  Let $\alpha\in\Self_d$. In accordance with
  Lemma~\ref{self-mod-factoring}, write $\alpha = 1B_12B_2\ldots kB_k$,
  in which $k=\max\alpha$ and each factor $iB_i$ is decreasing with pace
  $d$. Fixing $i\in [k]$, assume that the length of $iB_i$ is $\ell$ and write
  \[
    iB_i=c_1c_2\dots c_{\ell}.
  \]
  Let $m_j=c_j-c_{j+1}$ for $j\in [\ell-1]$ and let $m_{\ell}=c_{\ell}$.
  Note that $iB_i=c_1c_2\dots c_{\ell}$ can be recovered from the composition
  \[
    \mu = (m_1,m_2,\dots,m_{\ell}) \vDash i
  \]
  by letting $c_j=m_j+\cdots+m_{\ell}$. In other words, $\mu$
  encodes the factor $iB_i$, and the length of $iB_i$ equals the number of parts of $\mu$.
  If we exclude the last part of $\mu$ and
  let $\mu'=(m_1,\dots,m_{\ell-1})$, then we have a composition each
  part of which is at least $d$, and the
  possible choices for $\mu'$ are recorded by $K_{d,m}(x)$ as defined 
in~\eqref{KdnxDef}, where
  $m=m_1+\cdots+m_{\ell-1}$.
  The last part, $m_{\ell}$, of $\mu$ can be
  any integer between $1$ and $i$ and hence the possible choices for
  $\mu$ are recorded by
  \[
    xK_{d,0}(x) + xK_{d,1}(x) + \cdots + xK_{d,i-1}(x).
  \]
  By item~(b) of Lemma~\ref{K-F}, this expression simplifies to $xF_{d,i-1}(x)$ and
  hence the generating function for the set of $\alpha\in\Self_d$ with
  $k=\max\alpha$ is
  \[
    q^k\prod_{i=1}^k xF_{d,i-1}(x) = q^kx^k F^{!}_{d,k-1}(x).
  \]
  The claim now follows by summing over $k$.
\end{proof}

Theorem~\ref{self-gf} allows us to easily tabulate the cardinalities of
$\Self_{d,n}$; see Table~\ref{table-selfmod}.
\begin{table}
\[
\begin{array}{c|rrrrrrrrrrrrr}
  d\hspace{2pt}\backslash\hspace{1pt} n \rule[-0.9ex]{0pt}{0pt}
  & 0 & 1 & 2 & 3 & 4 & 5 & 6 & 7 & 8 & 9 & 10 & 11 & 12 \\
  \hline \rule{0pt}{2.6ex}
  0 & 1 & 1 & 2 & 5 & 14 & 43 & 143 & 510 & 1936 & 7775 & 32869 & 145665 & 674338 \\
  1 & 1 & 1 & 1 & 2 & 4 & 10 & 27 & 81 & 262 & 910 & 3363 & 13150 & 54135 \\
  2 & 1 & 1 & 1 & 1 & 2 & 4 & 9 & 23 & 64 & 194 & 629 & 2177 & 7982 \\
  3 & 1 & 1 & 1 & 1 & 1 & 2 & 4 & 9 & 22 & 58 & 167 & 515 & 1698 \\
  4 & 1 & 1 & 1 & 1 & 1 & 1 & 2 & 4 & 9 & 22 & 57 & 158 & 467 \\
  5 & 1 & 1 & 1 & 1 & 1 & 1 & 1 & 2 & 4 & 9 & 22 & 57 & 157 \\
  6 & 1 & 1 & 1 & 1 & 1 & 1 & 1 & 1 & 2 & 4 & 9 & 22 & 57 \\
\end{array}
\]
\caption{The number of self-modified $d$-ascent sequences of length $n$}
\label{table-selfmod}
\end{table}

For $d=0$ and $d=1$ we even get explicit general expressions. For $d=0$
we rediscover a formula originally given by Bousquet-Mélou~{\etal}~\cite{BMCDK:2fp}:
\begin{align*}
  \Self_0(q,x)
  &= 1 + qxF^{!}_0(x,qx) \\
  &= 1 + \sum_{k\geq 1}(qx)^k F^{!}_{0,k-1}(x) = \sum_{k\geq 0}\frac{(qx)^k}{(1-x)^{\binom{k+1}{2}}}.
\intertext{For $d=1$ (self-modified weak ascent sequences) we get the following
formula:}
  \Self_1(q,x)
  &= 1 + qxF^{!}_1(x,qx) \\
  &= 1 + \sum_{k\geq 1}(qx)^k F^{!}_{1,k-1}(x) = \sum_{k\geq 0}(qx)^k(1+x)^{\binom{k}{2}}.
\end{align*}

Note that the diagonals of Table~\ref{table-selfmod} each appear
to tend to some constant. To bring out this pattern let us skip the
first $d$ ones of each row and left adjust. In other words we are
interested in the coefficients of
$\bigl(\Self_d(1,x)-[d\kern 0.1em]_x\bigr)/x^d$, where
$[d]_x=1+x+\cdots+x^{d-1}$; we display those coefficients in
Table~\ref{table-selfmod-adjusted}.
Additional data (not displayed in the table) even suggests that the
$q$-analogue $(\Self_d(q,x)-[d\kern 0.1em]_{qx})/(qx)^d$
also tends to some polynomial in $q$, and we show in the next theorem that
this is the case.
What can be said about the sequence that is emerging as $d\to \infty$?
Formally, this limit should be understood as follows.
Let $G_0(x)$, $G_1(x)$, $G_2(x)$, \dots be a sequence of power
series. Then
\[
  \lim_{k\to\infty} G_k(x) = \sum_{n\geq 0}c_nx^n
\]
if, given a nonnegative integer $n$, there is a corresponding $K$
such that the coefficient of $x^n$ in $G_k(x)$ is $c_n$ for all
$k\geq K$.

\begin{table}
\[
\begin{array}{c|rrrrrrrrrrrrr}
  d\hspace{2pt}\backslash\hspace{1pt} n \rule[-0.9ex]{0pt}{0pt}
  & 0 & 1 & 2 & 3 & 4 & 5 & 6 & 7 & 8 & 9 & 10 & 11 & 12 \\
  \hline \rule{0pt}{2.6ex}
  0& 1& 1& 2& 5& 14& 43& 143& 510& 1936& 7775& 32869& 145665& 674338 \\
  1& 1& 1& 2& 4& 10& 27& 81& 262& 910& 3363& 13150& 54135& 233671 \\
  2& 1& 1& 2& 4& 9& 23& 64& 194& 629& 2177& 7982& 30871& 125402 \\
  3& 1& 1& 2& 4& 9& 22& 58& 167& 515& 1698& 5925& 21810& 84310 \\
  4& 1& 1& 2& 4& 9& 22& 57& 158& 467& 1474& 4934& 17448& 64847 \\
  5& 1& 1& 2& 4& 9& 22& 57& 157& 454& 1387& 4476& 15243& 54606 \\
  6& 1& 1& 2& 4& 9& 22& 57& 157& 453& 1369& 4321& 14293& 49570 \\
  7& 1& 1& 2& 4& 9& 22& 57& 157& 453& 1368& 4297& 14027& 47615 \\
  8& 1& 1& 2& 4& 9& 22& 57& 157& 453& 1368& 4296& 13996& 47178 \\
  9& 1& 1& 2& 4& 9& 22& 57& 157& 453& 1368& 4296& 13995& 47139 \\
 10& 1& 1& 2& 4& 9& 22& 57& 157& 453& 1368& 4296& 13995& 47138
\end{array}
\]
\caption{The coefficients of $\bigl(\Self_d(1,x)-[d\kern 0.1em]_x\bigr)/x^d$}
\label{table-selfmod-adjusted}
\end{table}

\begin{theorem}\label{limit}
  We have
  \[
    \lim_{d\to\infty} \frac{1}{(qx)^d}\biggl(\Self_d(q,x)-[d\kern 0.1em]_{qx}\biggr)
    \,=\, \sum_{k\geq 0} (qx)^k(1+x)(1+2x)\cdots(1+kx).
  \]
\end{theorem}

Let $R(q,x)$ be the generating function on the right-hand side of the
claimed identity, and let $r_n(q)$ be the coefficient of
$x^n$ in $R(q,x)$. The sequence $(r_n(1))_{n\geq 0}$ starts
\[
  1,1,2,4,9,22,57,157,453,1368,4296,13995,47138,163779,585741,
\]
which agrees with the last row of
Table~\ref{table-selfmod-adjusted}. The generating function $R(1,x)$
defines entry A124380 in the OEIS. No combinatorial
interpretation is given for that entry, but it is now straightforward to
give one. We give this combinatorial interpretation in the next lemma
and we shall use it in the proof of Theorem~\ref{limit}.

\begin{lemma}
  Let $\R_n$ be the set of restricted growth functions
  $\alpha = 1B_12B_2\ldots kB_k$ such that, for each $i\in[k]$, either
  $B_i$ is empty or $B_i=b_i$ is a single letter with $b_i\leq i$. With
  $r_n(q)$ defined as above, we have
  \[
    r_n(q) = \sum_{\alpha\in\R_n} q^{\max\alpha}.
  \]
\end{lemma}

\begin{proof}
  The following generating function for the Stirling numbers of the
  second kind, $S(n,k) = \#\{\alpha\in\RGF_n: \max\alpha = k\}$, is
  well-known~\cite[equation 1.94c]{Stanley:ec1}:
  \[
    \sum_{n\geq 0} x^n \sum_{\alpha\in\RGF_n}q^{\max\alpha}
    = \sum_{k\geq 0}\frac{(qx)^k}{(1-x)(1-2x)\cdots(1-kx)}.
  \]
  The usual proof goes as follows. Let $\alpha \in \RGF_n$ with $\max\alpha = k$
  be given. Factor $\alpha = 1B_12B_2\ldots kB_k$ in which each letter
  of $B_i$ is at most $i$. The possible choices for $B_i$ are counted by
  the generating function $1/(1-ix)=1+ix+i^2x^2+\cdots$ and hence the
  possible choices for $\alpha$ such that $\max\alpha = k$ are counted by
  $x^k/((1-x)(1-2x)\cdots(1-kx))$. For $\alpha \in \R_n$ with
  $\max\alpha = k$ the same approach applies, but now $|B_i|\leq 1$ and
  the choices for $B_i$ are counted by $1+ix$.
\end{proof}

\begin{proof}[Proof of Theorem~\ref{limit}]
  Letting $f_{d,n}(q)$ denote the coefficient of $x^n$
  in $\Self_d(q,x)$ we find that the coefficient of $x^n$ in
  $(\Self_d(q,x)-[d\kern 0.1em]_{qx})/(qx)^d$ is
  $f_{d,n+d}(q)/q^d$.
  Since the limit to be proved has $d$ going to $\infty$ it suffices to show that, for any $d\geq 0$ and $d\ge n$,
  \[
  f_{d,n+d}(q) = q^dr_n(q).
  \]
  Assume $d\geq n$ and let $\alpha\in\Self_{d,n+d}$ with
  $k=\max\alpha$ be given. Write $\alpha=1B_12B_2\ldots kB_k$, where
  each $iB_i$ is decreasing with pace $d$. Note that $B_i$ is empty
  whenever $i\leq d$. Moreover, 
  since $n\le d$ we have that
  $B_{d+i}$ is empty or $B_{d+i}=b_{d+i}$
  is a singleton with $b_{d+i}\in [i]$ for $i\geq 1$. The desired
  bijection from $\Self_{d,n+d}$ onto $\R_n$ is now provided by
  \[
    1\,2\,\ldots\, d\, (d+1)B_{d+1}\,(d+2)B_{d+2}\, \ldots\, kB_k
    \;\;\mapsto\;\; 1B_{d+1}\,2 B_{d+2}\, \ldots \,(k-d)B_k.
  \]
  For instance, if $d=4$ and $n=3$, then
  \[1234516 \mapsto 112,\;\;
    1234561 \mapsto 121,\;\;
    1234562 \mapsto 122,\;\;
    1234567 \mapsto 123.
  \]
  Finally, if $\alpha \mapsto \beta$ in this way, then
  $\max\alpha=d+\max\beta$, which conludes the proof.
\end{proof}

\section{Self-modified \texorpdfstring{$d$}{d}-Fishburn permutations}

Let $\Sym$ denote the set of all permutations 
$\pi=p_1 p_2\ldots p_n$
of $[n]$ for all $n\ge0$.
Inspired by a question of Dukes and Sagan~\cite{DS:das}, Zang and
Zhou~\cite{ZZ:dperm} have recently introduced the set $\F_d$ of
\emph{$d$-Fishburn permutations}. They are a generalization of
\emph{Fishburn permutations}~\cite{BMCDK:2fp}
where the recursive structure of $\F_d$ is encoded by $\Ascseq_d$,
for every $d\ge 0$. The bijection between $\F_d$ and $\Ascseq_d$
defined this way is denoted by
$$
\Phi_d:\dA\to\F_d.
$$
From now on, we denote by $\Sym(p)$ the set of permutations
\emph{avoiding} the pattern~$p$ (see Bevan’s note~\cite{Bevan:pat}
for a brief introduction to the permutation patterns field).
Similarly, we denote by $U(p)$ the set of elements of $U$ avoiding $p$, where $U$ is a generic set.
Fishburn permutations are defined as the set $\F_0=\Sym(\fishpattern)$
of permutations avoiding~$\fishpattern$,
where $\fishpattern$ is the following (bivincular) mesh pattern~\cite{BMCDK:2fp, BC2011}:
$$
\fishpattern=
\begin{tikzpicture}[scale=0.35, baseline=17.5pt]
\fill[NE-lines] (1,0) rectangle (2,4);
\fill[NE-lines] (0,1) rectangle (4,2);
\draw (0.001,0.001) grid (3.999,3.999);
\filldraw (1,2) circle (5pt);
\filldraw (2,3) circle (5pt);
\filldraw (3,1) circle (5pt);
\end{tikzpicture}\ .
$$
That is, $\F_0$ is the set of all $\pi\in\Sym$ which do not contain a
subsequence $p_i p_j p_k$ with $j=i+1 <k$ and $p_k+1=p_i<p_j$.

We~\cite{CCS:mod} gave an alternative definition of $\F_d$ that
is reminiscent of the classical case $d=0$, which we now recall.
First, we describe a procedure to determine the \emph{$d$-active}
and \emph{$d$-inactive}~\cite{ZZ:dperm}
elements of a given permutation $\pi=p_1\ldots p_n$.
Let $\pi^{(k)}$ denote the restriction of $\pi$ to the elements of $[k]$.
\begin{itemize}
\item Set $1$ to be a $d$-active element.
\item For $k=2,3,\dots,n$, let $k$ be $d$-inactive if $k$ is to the
left of $k-1$ in $\pi$ and there exist at least $d$ elements of
$\pi^{(k)}$ between $k$ and $k-1$ that are $d$-active.
Otherwise, $k$ is said to be $d$-active.
\end{itemize}
Then, we say that a permutation~$\pi$ contains the \emph{$d$-Fishburn
pattern}~\cite{CCS:mod}, denoted by $\fishpattern_d$, if it contains an occurrence
$p_ip_jp_k$ of $\fishpattern$ where $p_i$ is $d$-inactive.
The other two elements $p_j$ and $p_k$ can be either $d$-active
or $d$-inactive.

As an example, let $\pi=2573164$ and $d=2$. We compute the $d$-active
elements of $\pi$, where such elements are set in boldface. It is easy to
see that the elements $k\le 4$ are $d$-active. Then, we have
$$
\pi^{(5)} = {\bf 2}\;5\;{\bf 3}\;{\bf 1}\;{\bf 4}.
$$
There are two $d$-active elements between~$5$ and~$4$, thus~$5$ is
$d$-inactive. Similarly, $6$ is $d$-active (since $6$ is to
the right of~$5$), but there are two $d$-active elements in
$$
\pi^{(7)} = {\bf 2}\;5\;7\;{\bf 3}\;{\bf 1}\;{\bf 6}\;{\bf 4}
$$
between~$7$ and~$6$, hence~$7$ is $d$-inactive. Finally, we see that
$\pi$ contains an occurrence of $\fishpattern_d$ realized by the elements
$5,7,4$.

With slight abuse of notation, we let
$\Sym(\fishpattern_d)$ be the set of permutations that do
not contain~$\fishpattern_d$. Finally, for every $d\ge 0$ we have
that~\cite[Prop.\ 7.1]{CCS:mod}
$$
\F_d=\Sym(\fishpattern_d).
$$

In the same paper, we showed that $d$-Fishburn permutations
can be alternatively obtained as the bijective image of
$d$-ascent sequences under the composition of the $d$-hat map with
the \emph{Burge transpose}~\cite{CC:tpb}, lifting a classical result by
Bousquet-M\'{e}lou~{\etal}~\cite{BMCDK:2fp} to any $d\ge 0$.
As we will build on this alternative description of $\F_d$,
we wish to recall the necessary tools and definitions.

Given an endofunction $\al=a_1\ldots a_n$, let
$$
\wDes\al =\{i\ge 2 \mid a_i \le a_{i-1} \}
$$
denote the set of \emph{weak descents} of $\al$.
The set of \emph{Burge words}~\cite{CC:tpb} is defined as
$$
\Bur_n=\left\lbrace
\binom{u}{\al}:u\in\WI_n,\;\al\in\Cay_n,\;\wDes(u)\subseteq\wDes(\al)
\right\rbrace,
$$
where $\binom{u}{\al}$ is a \emph{biword} (i.e.\ a pair of words of
the same length written one above each other)
and $\WI_n$ is the subset of $\Cay_n$ consisting of the weakly
increasing Cayley permutations. 
The \emph{Burge transpose} is a transposition operation~$T$ on~$\Bur_n$
defined as follows. Given $w=\binom{u}{\al}\in\Bur_n$, to compute $w^T$, first
turn each column of $w$ upside down. Now sort the columns in
weakly increasing order with respect to the top entry so that the entries below the elements in the upper row of a given value form a weakly decreasing sequence.
For instance, we have
$$
\binom{1\ 2\ 3\ 4\ 5\ 6\ 7\ 8\ 9}{1\ 4\ 1\ 2\ 5\ 2\ 2\ 3\ 2}^{\!T} \\[1ex]
= \binom{1\ 1\ 2\ 2\ 2\ 2\ 3\ 4\ 5}{3\ 1\ 9\ 7\ 6\ 4\ 8\ 2\ 5},
$$
since the two $1$'s in the bottom row of the first array become the  first two elements in the top of the transposed array, and the corresponding $1$ and $3$ in the top row of the first array get sorted into the decreasing sequence $31$ in the transpose, etc.
Observe that $T$ is an involution on $\Bur_n$. Furthermore,
by picking $\identity_n=12\ldots n$ as the top row, we obtain
a map $\burget:\Cay_n\to\Sym_n$ defined by\footnote{The map
$\burget$ was originally~\cite{CC:tpb} denoted by the letter $\ga$.}
$$
\binom{\identity_n}{\al}^{\!T}=\binom{\sort(\al)}{\burget(\al)},
$$
for any $\al\in\Cay$, where $\sort(\al)$ is obtained by sorting the
entries of $\al$ in weakly increasing order.
As a special case, if $\al$ is a permutation, then $\burget(\al)=\al^{-1}$;
note that this proves that $\burget$ is surjective.
Finally, for every $d\ge 0$ we have~\cite[Thm.\ 7.10]{CCS:mod}
$$
\Phi_d=\burget\circ\,\dhat.
$$
In other words, the diagram
\begin{equation}\label{triangle_any_d}
\begin{tikzpicture}[scale=0.4, baseline=20.5pt]
  \matrix (m) [matrix of math nodes,row sep=3.5em,column sep=7em,minimum width=2em]
  {
    \kAscseq{d} & \F_d  \\
            & \kModasc{d}    \\
  };
  \path[-stealth, semithick]
  (m-1-1) edge node [above, yshift=2pt] {$\Phi_d$} (m-1-2)
  (m-1-1) edge node [below, xshift=-10pt]{$\dhat$} (m-2-2)
  (m-2-2) edge node [right] {$\burget$} (m-1-2);
\end{tikzpicture}
\end{equation}
commutes for every $d\ge 0$ and all the arrows are size-preserving
bijections.

As an illustration, let $\al=1124253$. Note that $\al\notin\kAscseq{0}$,
but $\al\in\kAscseq{d}$ for $d\ge 1$. The $2$-Fishburn permutation
associated with $\al$ under $\Phi_2$ is $\pi=\Phi_2(\al)=2537146$, but we
shall omit the details here. Instead, we follow the bottom side of
diagram~\eqref{triangle_any_d} and show that
$\burget\circ\khat{2}(\al)=\pi$. To start, we compute $\khat{2}(\al)$.
We have $\Asc_2\al=(1,2,3,4,6)$, so that:
\begin{align*}
  \al              &= 1124253;\\
  M(\al,1)         &= \underline{1}124253;\\
  M(\al,1,2)       &= \mathbf{2}\underline{1}24253;\\
  M(\al,1,2,3)     &= \mathbf{3}1\underline{2}4253;\\
  M(\al,1,2,3,4)   &= 312\underline{4}253;\\
  M(\al,1,2,3,4,6) &= 31242\underline{5}3=\khat{2}(\al).
\end{align*}
Finally, we apply the Burge transpose
$$
\binom{\identity_7}{\al}^{\!T}
=\binom{1\ 2\ 3\ 4\ 5\ 6\ 7}{3\ 1\ 2\ 4\ 2\ 5\ 3}^{\!T} \\[1ex]
= \binom{1\ 2\ 2\ 3\ 3\ 4\ 5}{2\ 5\ 3\ 7\ 1\ 4\ 6}
$$
and obtain $\pi$ as the bottom row of the resulting biword.

For the rest of this section, let $d\ge 0$ be a fixed nonnegative
integer. Denote by
$$
\SMF_d=\Phi_d(\Self_d)
$$
the set of $d$-Fishburn permutations corresponding to
self-modified $d$-ascent sequences under the bijection $\Phi_d$.
Bousquet-Mélou {\etal}~\cite{BMCDK:2fp} proved 
a beautiful characterization of $\SMF_0$ in terms of pattern avoidance that we wish to generalize to all $d$,
In particular, when $d=0$, they showed
$\SMF_0=\Sym(3\bar{1}52\bar{4})$.
Here a permutation $\pi$ avoids the barred pattern
$3\bar{1}52\bar{4}$ if every occurrence of the pattern $231$
plays the role of $352$ in an occurrence of $31524$.
Formally, for every $i<j<k$ such that $p_k<p_i<p_j$,
there exist $\ell\in\{i+1,i+2,\dots,j-1\}$ and $m>k$ such that
$p_ip_{\ell}p_jp_kp_m$ is an occurrence of $31524$.
More visually, in terms of mesh patterns,
$$
\SMF_0=\Sym\left(
\begin{tikzpicture}[scale=0.4, baseline=19pt]
\fill[NE-lines] (1,0) rectangle (2,1);
\draw [semithick] (0,1) -- (4,1);
\draw [semithick] (0,2) -- (4,2);
\draw [semithick] (0,3) -- (4,3);
\draw [semithick] (1,0) -- (1,4);
\draw [semithick] (2,0) -- (2,4);
\draw [semithick] (3,0) -- (3,4);
\filldraw (1,2) circle (5pt);
\filldraw (2,3) circle (5pt);
\filldraw (3,1) circle (5pt);
\end{tikzpicture},
\begin{tikzpicture}[scale=0.4, baseline=19pt]
\fill[NE-lines] (3,2) rectangle (4,3);
\draw [semithick] (0,1) -- (4,1);
\draw [semithick] (0,2) -- (4,2);
\draw [semithick] (0,3) -- (4,3);
\draw [semithick] (1,0) -- (1,4);
\draw [semithick] (2,0) -- (2,4);
\draw [semithick] (3,0) -- (3,4);
\filldraw (1,2) circle (5pt);
\filldraw (2,3) circle (5pt);
\filldraw (3,1) circle (5pt);
\end{tikzpicture}
\right).
$$
In the same spirit, we wish to characterize $\SMF_d$.
From now on, we say that a permutation $\pi=p_1\ldots p_n$
\emph{contains} $\selfpat_d$ if there are two indices $i<j$ such that
\begin{itemize}
\item $p_i=p_j+1$; and
\item $\asc(p_ip_{i+1}\ldots p_j)\leq d$.
\end{itemize}
Furthermore, with slight abuse of notation, we write
$\Sym(\selfpat_d)$ to denote the set of permutations that
do not contain~$\selfpat_d$. Our goal is to prove that
\begin{equation}\label{eq_SMF_patts}
\SMF_d=\Sym(3\bar{1}52\bar{4},\selfpat_d)=\SMF_0(\selfpat_d).
\end{equation}
Observe that
\begin{equation}\label{FtA}
\SMF_d=\Phi_d(\Self_d)
=\burget\bigl(\dhat(\Self_d)\bigr)
=\burget(\Self_d),
\end{equation}
where $\dhat(\Self_d)=\Self_d$ by definition of self-modified
$d$-ascent sequence. In other words, permutations in $\SMF_d$
are precisely those that are obtained by applying the Burge transpose
to some self-modified $d$-ascent sequence in $\Self_d$. 
Furthermore, by Lemma~\ref{self-mod-factoring} every
$\al\in\Self_d$ can be written $\al=1B_12B_2\ldots kB_k$,
where $k=\max\al$ and each factor $iB_i$ is decreasing with
pace~$d$.

Since we know from Proposition~\ref{selfmod_rgf_char} that
$\Self_d\subseteq\RGF$, it will be useful to describe the ascents of
permutations in $\burget(\RGF)$.
\begin{lemma}\label{asc_of_tRGF}
Let $\al\in\RGF$, $k=\max\al$, and
$$
\binom{\id}{\al}^T=\binom{\sort(\al)}{\pi}
=\begin{pmatrix}
1\dots 1 & 2\dots 2 & \ldots & k\dots k\\
p_1\dots p_{i_1} & p_{i_1+1}\dots p_{i_2} & \ldots & p_{i_{k-1}+1}\dots p_{i_k}
\end{pmatrix}.
$$
Then:
\begin{enumerate}
\item[(a)] $\Asc\pi=\{1,i_1+1,i_2+1,\dots,i_{k-1}+1\}$.
\item[(b)] For each $i<j$, we have
$$
\asc(p_i\dots p_j)=\ell_j-\ell_i+1,
$$
where $\ell_i$ and $\ell_j$ are the entries above $p_i$ and $p_j$
in $\sort(\al)$, respectively.
\end{enumerate}
\end{lemma}
\begin{proof}
To prove (a), let $j\in[k]$. By definition of Burge transpose,
the entry $p_{i_j}$ below the rightmost copy of~$j$ in $\sort(\al)$ is
the index of the leftmost copy of~$j$ in~$\al$.
Since $\al\in\RGF$, the leftmost copy of~$j$ precedes each copy
of~$j+1$ in~$\al$; that is, $p_{i_j}<p_{i_j+1}$ for each $j\in[k-1]$,
and thus $\Asc\pi\supseteq\{1,i_1+1,i_2+1,\dots,i_{k-1}+1\}$.
The other inclusion is trivial because each factor
$p_{i_j+1}\dots p_{i_{j+1}}$ is weakly decreasing by definition
of Burge transpose.

We shall now prove (b). By the first item, the ascents of~$\pi$
are precisely those entries that are below the leftmost copy of
some $j\in[k]$ in~$\sort(\al)$.
Since the first position is an ascent by definition,
the sequence $p_i\dots p_j$ contains exactly one ascent for each
number from $\ell_i$ (the entry above $p_i$) to $\ell_j$ (the entry above $p_j$); that is, we have
\[
\asc(p_i\dots p_j)=\ell_j-\ell_i+1.\qedhere
\]
\end{proof}

\begin{theorem}\label{thm_dselfperm}
We have
$$
\SMF_d = \SMF_0(\selfpat_d).
$$
\end{theorem}
\begin{proof}
A special case of item~(a) of Theorem~\ref{dself_zero} is 
that $\Self_d\subseteq\Self_0$. Using this and equation~\eqref{FtA} we obtain
$$
\SMF_d = \burget(\Self_d)
\subseteq\burget(\Self_0)
=\SMF_0.
$$
Now, let $\pi\in\SMF_0$.
Since $\SMF_0=\burget(\Self_0)$, there is some $\al\in\Self_0$
such that
$$
\binom{\id}{\al}^T=\binom{\sort(\al)}{\pi}.
$$
Furthermore, by Lemma~\ref{self-mod-factoring},
we have
$$
\al=1B_12B_2\dots kB_k,
$$
where $k=\max\al$ and each factor $iB_i$ is decreasing with pace zero,
i.e.\ is weakly decreasing.
To obtain the desired equality, it suffices to prove that
$\pi$ contains $\selfpat_d$ if and only if $\pi\notin\SMF_d$.

Let us start with the forward direction.
By definition, $\pi$ containing $\selfpat_d$ means that there are indices $i<j$ with
$p_i=p_j+1$ and $\asc(p_i\dots p_j)\le d$.
Now applying item~(b) of Lemma~\ref{asc_of_tRGF} (which can be done since $\alpha$ is an RGF)
shows that the ascent inequality implies $\ell_j-\ell_i<d$ where we are using the notation of the lemma. But $p_i=p_j+1$ so that
$$
\binom{\id}{\al}=\binom{\dots p_j\ p_i \dots}{\dots \ell_j\ \ell_i\dots}.
$$
Since $i<j$ we have that $\ell_j\ge\ell_i$ and this forces the pair
$\ell_j\ell_i$ to be part of the same factor $k B_k$ above.
On the other hand $\ell_j-\ell_i<d$ so that factor cannot be decreasing with pace $d$.
Using Lemma~\ref{self-mod-factoring} again, we see that $\al\notin\Self_d$.
Thus, since $\burget$ is injective, $\pi=\burget(\al)\notin\burget(\Self_d)=\SMF_d$ as desired.

On the other hand, suppose that $\pi\notin\SMF_d$, i.e.\
$\al\notin\Self_d$. Therefore, at least one of the blocks
$B_1,\ldots,B_k$ of $\al$ (which are weakly decreasing since
$\al\in\Self_0$) is not decreasing with pace~$d$.
Let $a_{h}a_{h+1}$ be two consecutive elements in $B_i$,
$1\le i\le k$, with $a_{h}-a_{h+1}<d$.
Note that $a_{h}\ge a_{h+1}$ since $B_i$ is weakly decreasing.
Now, applying the Burge transpose yields
$$
\binom{\id}{\al}^T
=\begin{pmatrix}
\dots & h & h+1 & \dots\\
\dots & a_{h} & a_{h+1} & \dots
\end{pmatrix}^T
=\begin{pmatrix}
\dots & a_{h+1} & \dots & a_{h} & \dots\\
\dots & h+1 & \dots & h & \dots
\end{pmatrix}=\binom{\sort(\al)}{\pi}.
$$
Finally, the elements $h+1$ and $h$ form an occurrence of
$\selfpat_d$ in $\pi$ since by item~(b) of Lemma~\ref{asc_of_tRGF}
the number of ascents contained between them is equal to
$$
a_{h}-a_{h+1}+1<d+1\le d.
$$
This concludes the proof.
\end{proof}

\section{Pattern avoidance in \texorpdfstring{$\SMF_d$}{SMF}}

The first two authors have developed a theory of transport
of patterns between Fishburn permutations and modified ascent
sequences~\cite{CC:tpb}. We recently showed that the same machinery
also applies to $d$-Fishburn permutations and
modified $d$-ascent sequences~\cite[Thm.\ 8.6]{CCS:mod}.
More explicitly, for any $d\ge 0$ and permutation~$\tau$,
$$
\burget:\dModasc(B_{\tau})\longrightarrow\F_d(\tau)
$$
is a size-preserving bijection, where $B_{\tau}$ is a finite set called
the \emph{Fishburn basis} of~$\tau$, $\dModasc(B_{\tau})$ is the set
of modified $d$-ascent sequences avoiding every pattern in $B_{\tau}$,
and $\F_d(\tau)$ is the set of $d$-Fishburn permutations avoiding $\tau$.
Moreover, there is a constructive procedure to compute~$B_{\tau}$~\cite{CC:tpb}.

Since the map~$\burget$ is injective on
$\dModasc$~\cite[Cor.\ 4.7]{CCS:mod}, it is also injective on
$\Self_d=\dA\cap\dModasc$. The following is a
corollary to the previous discussion and the transport
theorem~\cite[Thm.\ 8.6]{CCS:mod}.

\begin{corollary}\label{transport_thm}
  For any $d\ge 0$ and permutation~$\tau$, the map
  $
  \burget:\Self_d(B_{\tau})\rightarrow\SMF_d(\tau)
  $
  is a size-preserving bijection.
\end{corollary}

For instance, let us consider the pattern $213$.
Recall~\cite{CC:tpb} that $B_{213}=\{112,213\}$ and hence
$\#\SMF_d(213)=\#\Self_d(112,213)$. 
It is a fundamental problem in the field of pattern avoidance to enumerate avoidance classes of permutations.
We will enumerate the class just mentioned below.
Interestingly, the generating function for this subset of self-modified $d$-ascent sequences can be expressed in terms of the same polynomials which enter into the computation of 
$\tilde{A}_d(q,x)$ for all self-modified $d$-ascent sequences without restriction; see Theorem~\ref{self-gf}.

\begin{lemma}\label{112-213-factoring}
  Let $\alpha\in\Self_d$ and write $\alpha = 1B_12B_2\ldots kB_k$ as in
  Lemma~\ref{self-mod-factoring}. Then $\alpha$ avoids $112$ and $213$
  if and only if $B_1=B_2=\cdots =B_{k-1}=\emptyset$.
\end{lemma}

\begin{proof}
  Assume that $B_1=B_2=\cdots =B_{k-1}=\emptyset$ so that
  $\alpha = 12\ldots kB_k$. Since the last two letters of $112$ form an
  ascent and $kB_k$ is decreasing (with pace $d$), any occurrence of
  $112$ would have to wholly reside in the prefix $12\ldots k$ of
  $\alpha$, but that is clearly impossible. Similarly, $\alpha$ avoids
  $213$. To prove the converse, assume that $B_i\neq \emptyset$ for some
  $i\in [k-1]$ and let $b$ be the first letter of $B_i$. If $b = i$
  (this can only happen if $d=0$), then $ibk$ is an occurrence of $112$
  in $\alpha$. Otherwise, $b < i$ and $ibk$ is an occurrence of $213$ in
  $\alpha$.
\end{proof}

Recall the generating function $F_{d,n }(x)$ as defined in~\eqref{Fdnx}.

\begin{proposition}\label{des-self-distrib}
  We have
  \[
    F_{d+1,n-1}(q) =
    \sum_{\al\in\Self_{d,n}(112,213)}q^{\wdes\al} =
    \sum_{\pi\in\SMF_{d,n}(213)}q^{\ides\pi},
  \]
  where $\wdes\al=\#\wDes\al$ is the number of weak descents of $\al$
  and $\ides\pi = \des(\pi^{-1})$ is the number of inverse descents
  of $\pi$.
\end{proposition}

\begin{proof}
  Let $\al\in\Self_d(112,213)$, $k=\max\al$ and write
  $\alpha = 12\ldots kB_k$ as in Lemma~\ref{112-213-factoring}. By
  following the same reasoning as in the proof of Theorem~\ref{self-gf},
  we find that
  \begin{equation*}
    \sum_{\al}q^{\max\al}x^{|\al|}
    = 1 + \sum_{k\geq 1} q^kx^k F_{d,k-1}(x) = 1 + qxF_{d}(x,qx),
  \end{equation*}
  in which $\al$ ranges over $\Self_{d}(112,213)$. Using the explicit
  formula~\eqref{dfib-gf} from Section~\ref{section_enumdself} it is easy to verify that
  $F_{d+1}(q,x) = F_{d}(qx,x)$. Also, by Lemma~\ref{self-mod-factoring}, $\wdes\al+\max\al = |\al|$.
  Putting this all together we have
  \begin{equation*}
    \sum_{\al}q^{\wdes\al}x^{|\al|}
    = \sum_{\al}q^{-\max\al}(qx)^{|\al|}
    = 1 + xF_d(qx, x)
    = 1 + xF_{d+1}(q, x).
  \end{equation*}
  The first equality of our proposition follows by identifying coefficients in this identity.
 
  Now, the Fishburn basis of $213$ is $\{112,213\}$ and, by
  Corollary~\ref{transport_thm}, the Burge transpose~$\burget$ is
  a bijection between $\Self_{d,n}(112,213)$ and $\SMF_{d,n}(213)$.
  More specifically~\cite[Lemma 5.2]{CC:caypol}, a sequence
  $\al\in\Self_{d,n}(112,213)$ with~$k$ weak descents is mapped
  under~$\burget$ to a permutation $\pi\in\SMF_{d,n}(213)$ with~$k$
  inverse descents, and hence the second equality follows.
\end{proof}

\paragraph*{Acknowledgments.} We would like to thank the
referees for their careful reading of the manuscript and helpful
suggestions, which have significantly improved the clarity of
the paper.

%

\bibliographystyle{alpha}

\end{document}